\documentclass[a4paper]{article}

\usepackage{graphicx}
\usepackage{amsmath,amsfonts,amssymb,amsthm,amscd,enumerate}
\usepackage[all]{xy}
\usepackage[active]{srcltx}

\usepackage{cite}

\newtheorem{thm}{Theorem} 
\newtheorem{prop}[thm]{Proposition}

\theoremstyle{definition}

\newcommand{\nc}[2]{\newcommand{#1}{#2}}
\newcommand{\rnc}[2]{\renewcommand{#1}{#2}}
\rnc{\[}{\begin{equation}}
\rnc{\]}{\end{equation}}
\nc{\wegengruen}{\end{equation}}

\newcommand{\Z}{\mathbb{Z}} 
\newcommand{\N}{\mathbb{N}}
\newcommand{\R}{\mathbb{R}}
\newcommand{\C}{\mathbb{C}}
\newcommand{\hsp}{{\hspace{-1pt}}}
\newcommand{\hs}{{\hspace{1pt}}}

\newcommand{\hH}{\mathcal{H}}

\newcommand{\cO}{\mathcal{O}}

\newcommand{\A}{{\mathcal{A}}}

\newcommand{\F}{{\mathcal{F}}}

\newcommand{\s}{\sigma}
\newcommand{\sa}{\sigma^\alpha}

\newcommand{\dd}{\mathrm{d}}

\newcommand{\im}{\mathrm{i}}
\newcommand{\E}[1]{\mathrm{e}^{\im #1}}  \newcommand{\e}{\mathrm{e}}

\newcommand{\gen}{\mathrm{span}}

\newcommand{\spec}{\mathrm{spec}}

\newcommand{\op}{\mathrm{op}}

\newcommand{\lN}{\ell_2(\N)}
\newcommand{\lZ}{\ell_2(\Z)}
\newcommand{\Tr}{{\mathrm{Tr}}}

\newcommand{\ra}{\rightarrow}
\newcommand{\lra}{\longrightarrow}

\newcommand{\ot}{\otimes}
\newcommand{\ob}{\,\bar\otimes\,}

\renewcommand{\S}{\mathbb{S}^1}
\def\SUq{\cO(\mathrm{SU}_q(2))}
\def\CSU{C(\mathrm{SU}_q(2))}
\def\SU{\mathrm{SU}(2)}
\newcommand{\Uq}{\cO(\mathrm{D}_q)}  
\newcommand{\CUq}{C(\mathrm{D}_q)}  
\newcommand{\UD}{\mathrm{D}}  
\newcommand{\dFq}{\F^{(1)}(\UD_q)} 
\newcommand{\Fq}{\F(\UD_q)} 
\newcommand{\ip}[2]{\langle{#1},{#2}\rangle}  
\newcommand{\dt}{\mbox{$\frac{\partial}{\partial t}$}}
\newcommand{\dz}{\mbox{$\frac{\partial}{\partial z}$}}
\newcommand{\dbz}{\mbox{$\frac{\partial}{\partial \bar z}$}}

\title{Representations of quantum SU(2) operators on a local chart}

\author{
{\sc Elmar Wagner\footnote{
{\it MSC2010:} 46L85, 46L52, 58B32  \ \  
{\it Key Words:} noncommutative function spaces, noncommutative differential calculus, quantum SU(2), Hilbert space representations}
} \\
\normalsize
Instituto de F\'isica y Matem\'aticas\\
\normalsize
Universidad Michoacana de San Nicol\'as de Hidalgo, Morelia, M\'exico\\
\normalsize
e-mail: {\it elmar@ifm.umich.mx}}

\usepackage{geometry}
\date{}   
\begin{document}
\maketitle

\begin{abstract}
Hilbert space representations of quantum SU(2) by multiplication operators on a 
local chart are constructed, where the local chart is given by tensor products 
of square integrable functions on a quantum disc and on the classical unit circle. 
The actions of generators of quantum SU(2), generators of the opposite algebra, 
and noncommutative partial derivatives are computed on a Hilbert space basis. 
\end{abstract}

\section{Introduction}
In \cite{KW}, a twisted Dirac operator for quantum $\SU$ is constructed by 
using so-called disc coordinates. The starting point of this construction is 
a representation of quantum $\SU$  by multiplication operators 
on the tensor product of $L_2(\S)$ with the quantum disc algebra, where the inner product on the 
latter space is defined by a positive weighted trace. The construction of the Dirac 
operator uses noncommutative first order differential operators satisfying 
a twisted Leibniz rule. Only this twisted Leibniz rule is needed to prove that 
the Dirac operator has bounded twisted commutators with differentiable 
functions, thus  avoiding an explicit description of the Hilbert space actions  
of the involved operators. 

On the other hand, an explicit description of 
the Hilbert space representation of the quantum algebra of differentiable functions 
and the action of the Dirac operator is crucial for determining 
the analytic properties of the twisted spectral triple, for instance 
the spectrum of the Dirac operator, its K-homology class, 
its eigenvectors, and the C*-closure of the quantum algebra. 
The purpose of the present work is to describe the actions 
of all operators occurring in \cite{KW} on a Hilbert space basis. 
Unfortunately, a detailed analysis of the twisted spectral triple from \cite{KW} 
is behind the scope of this paper. 

\section{Quantum SU(2) in disc coordinates} 

Let $q\in(0,1)$. The coordinate ring $\SUq$ of quantum $\SU$ is the *-algebra generated by $c$ and $d$ satisfying 
\begin{eqnarray}
& \label{rels1}
        cd=qdc,\qquad
        c^*d=qdc^*,\qquad
        cc^*=c^*c,&\\
&\label{rels2}
        d^*d+q^2cc^*=1,\qquad dd^*+ cc^*=1.  &
\end{eqnarray}
Classically, i.\,e.~$q=1$,  the universal C*-algebra generated by $c$ and $d$ is isomorphic to $C(\mathbb{S}^3)$. 
The universal C*-closure $\CSU$ of $\SUq$ has been studied in \cite{MNW} and \cite{Wo}. 
Here we will use the fact that it is generated by the operators $c,d\in B(\lN\ob\lZ )$ given by 
\[  \label{cd}
 c\hs (e_n\ot b_k) =  q^n \hs e_{n} \ot b_{k+1}\,, \qquad
 d\hs (e_n\ot b_k) = \sqrt{1 - q^{2(n+1)}} \hs e_{n+1} \ot b_k\, ,   
\]
where  $\{e_n\}_{n\in{\N}}$ and $\{b_k\}_{k\in\Z}$ are orthonormal bases for $\lN$ and $\lZ$, respectively. 
Note that $c$  and $d$ from \eqref{cd} do indeed satisfy the relations \eqref{rels1} and \eqref{rels2}. 
Moreover, the representation of $\SUq$ defined by \eqref{cd} is faithful. 

Let $u: \S\lra \S$,  \,$u(\E{t})=\E{t}$, denote the unitary generator of $C(\S)$. 
On the orthonormal  basis $\{ b_k:=\frac{1}{\sqrt{2\pi}} \E{kt} : k\in\Z\}$ for $L_2(\S)\cong \lZ$, multiplication by $u$ 
becomes the bilateral shift 
\[ \label{u}
u\hs b_k = b_{k+1}\,, \qquad k\in\Z. 
\]

Consider the bounded  operators $z$ and $z^*$ on $\lN$ given by 
\[ \label{z}
 z\hs e_n := \sqrt{1 - q^{2(n+1)}} \hs e_{n+1}\,, \qquad 
  z^*\hs e_n := \sqrt{1 - q^{2n}} \hs e_{n-1}\,. 
\]
Obviously,  $z$ and $z^*$ satisfy the quantum disc relation 
\[ \label{qd} 
    z^* \hs z - q^2 z\hs z^* = 1 - q^2. 
\]
The polynomial *-algebra generated by $z$ and $z^*$ will be denoted by $\Uq$. 
Here and subsequently, $\mathrm{D}:=\{ z\in\C:|z|<1\}$ and 
$\bar{\mathrm{D}}:= \{ z\in\C:|z|\leq 1\}$ stand for the open and closed unit discs, respectively, 
and $\Uq$ will be called quantum disc algebra. 
As shown in \cite{KL},  the C*-closure $\CUq$ of $\Uq$ is isomorphic to the Toeplitz algebra. 
Moreover, \eqref{z} determines the unique (up to unitary equivalence)  faithful irreducible *-representation of $\CUq$, see e.\,g.~\cite{KueW}.

Defining $y\in B(\lN)$ by 
\[   \label{y}
   y \hs e_n   := q^n \hs e_n\,, \quad n\in\N, 
\]
we have 
\[ \label{yz}
  y = \sqrt{1-zz^*}, \qquad yz=qzy, \qquad z^* y = qy z^*. 
\]
In particular,  $y = \sqrt{1-zz^*} \in \CUq$. Finally, with $u$, $z$ and $y$ given in \eqref{u}, 
 \eqref{z}  and  \eqref{y}, 
 the operators $c$ and $d$ in \eqref{cd} can be written
 \[ \label{cdot} 
 c = y \ot u = \sqrt{1-zz^*} \ot u ,\qquad d= z \ot 1. 
 \]
 
The representation \eqref{cdot} corresponds to the classical parametrization 
$$
\psi : \bar \UD \times [-\pi, \pi] \lra \mathbb{S}^3, \qquad \psi(z,t):= (z, \sqrt{1-z\bar z} \hs \e^{\im t}), 
$$
The restriction of $\psi$ to the open set $\UD\times (-\pi, \pi)$ defines a dense coordinate chart for 
$\mathbb{S}^3 \cong\SU$ which is compatible with the standard differential structure on $\mathbb{S}^3$. 
Moreover, the generators $c$ and $d$ correspond in the classical limit $q\ra 1$ to the coordinate functions 
$$
c(z, \sqrt{1-z\bar z} \hs \e^{\im t}) = \sqrt{1-z\bar z} \hs \e^{\im t}, \qquad d(z, \sqrt{1-z\bar z} \hs \e^{\im t}) = z. 
$$
We call them disc coordinates because the discs $\psi_t(\bar\UD):=\{ \psi(z,t) : z\in\bar\UD\}$ for  fixed  $t\in(-\pi, \pi] $
correspond to the 2-dimensional symplectic leaves associated to a known Poisson group structure on $\SU$\cite{CP}.

\section{Differential calculus} 

One of the  differential operators  occurring in \cite{KW} is the classical partial derivative $-\im \dt$ acting on $C^{(1)}(\S)$. 
On the basis vectors $b_k$ from \eqref{u}, one obviously has 
\[
 -\im \dt \hs b_k = k\hs b_k\,, \qquad k\in\Z. 
\]

Our next aim is to describe noncommutative partial derivatives  $\dz$ and $\dbz$ on the quantum disc algebra. 
As in \cite{SSV}, consider the first order differential *-calculus $\dd : \Uq \lra \Omega(\UD_q)$ 
given by  $\Omega(\UD_q) = \dd z\hs   \Uq +  \dd z^* \Uq$ with $\Uq$-bimodule structure 
$$
z\hs  \dd z = q^{-2} \dd z \hs z , \quad z^*\hs  \dd z = q^2  \dd z \hs  z^*, \quad 
z\hs  \dd z^* = q^{-2} \dd z^*\hs  z, \quad  z^* \dd z^* = q^2  \dd z^* \hs z^*, 
$$
satisfying the Leibniz rule $\dd(fg)= \dd(f)\hs g + f\hs \dd(g)$ for all $f,g\in \Uq$. 
We define the partial derivatives $\dz$ and $\dbz$ by 
$$
d(f) =  \dd z \hs  \dz(f)   +\dd z^*\hs  \dbz(f)  , \qquad f\in \Uq. 
$$
On monomials, one gets 
$\dd(z^nz^{*k}) =\sum_{j=0}^{n-1} q^{-2j} \dd z \hs z^{n-1} z^{*k} + \sum_{l=0}^{k-1} q^{-2n+2l} \dd z \hs z^{n} z^{*k-1}$, 
therefore 
$$
\dz (z^nz^{*k}) = q^{-2(n-1)} \hs \mbox{$\frac{1-q^{2n}}{1-q^2}$}  z^{n-1} z^{*k} ,\qquad \dbz(z^nz^{*k}) = q^{-2n}\hs  \mbox{$\frac{1-q^{2k}}{1-q^2}$}  z^{n} z^{*k-1}. 
$$
From \eqref{qd}, it follows that the monomials $z^nz^{*k}$ span the linear space $\Uq$. 
Using the facts that commutators satisfy the Leibniz rule, 
$[z,z]  = [z^*,z^*]  = 0$,  $\frac{1}{1-q^2}[z^*,z]  = y^2$, and $y^2$ commutes with $z$ and $z^*$ in the same way as $\dd z$ and $\dd z^*$ do, 
one readily verifies that 
\[ \label{dzdbz} 
\dz(f) = \mbox{$\frac{1}{1-q^2}$}\hs  y^{-2}\hs  [z^* , f],\qquad \dbz(f) = \mbox{$\frac{-1}{1-q^2}$} \hs y^{-2} \hs [z , f],\qquad f\in\Uq. 
\]
In order extend the partial derivatives to an algebra of differentiable functions, we first observe that, by \eqref{qd} and \eqref{yz}, each 
$f\in \Uq$ can be written 
\[ \label{poly}
f= \sum_{n=0}^N  z^n \hs p_n(y^2)  +  \sum_{n=1}^M p_{-n}(y^2) \hs z^{*n} ,\qquad N,M\in\N, 
\]
with polynomials $p_n$ in one variable. Furthermore, for all functions $f \in C(\spec\{y^2\})$, 
the formulas in \eqref{dzdbz} together with \eqref{yz} give 
\begin{align*}
\dz f(y^2) =  \mbox{$\frac{1}{1-q^2}$}\hs  y^{-2}\hs(z^* f(y^2) -f(y^2) z^*) = -\frac{f(y^2) -f(q^2 y^2)}{y^2-q^2y^2} \hs  z^* = -\nabla_{\! q^2} f(y^2)\hs  z^*,\\
\dbz f(y^2) =  \mbox{$\frac{-1}{1-q^2}$}\hs  y^{-2}\hs(z f(y^2)  - f(y^2)z) = -z\hs \frac{f(y^2) -f(q^2 y^2)}{y^2-q^2y^2}  = -z\hs \nabla_{\! q^2}f(y^2),\
\end{align*}
where
$$
\nabla_{\! q^2} f(x):= \frac{f(x) -f(q^2 x)}{x-q^2x}
$$ 
denotes the $q^2$-difference operator 
and the operators $f(y^2)$ are defined by the spectral calculus of the self-adjoint operator $y^2$. The reason why we prefer 
to consider functions  $f = f(y^2)$ instead of $f=f(y)$ is because $z\mapsto 1-z\bar z$ is a differentiable function on $\bar \UD$ but 
$z\mapsto \sqrt{1-z\bar z}$ is not. Note that 
$$
\spec(y^2) = \{ q^{2n}:n\in \N\} \cup \{0\} 
$$
and 
$\lim_{k\ra \infty} \frac{f(q^{2k}) -f(q^{2k+2})}{q^{2k}-q^{2k+2}} 
= \lim_{k\ra \infty}\big( \frac{1}{1-q^2}\frac{f(q^{2k}) -f(0)}{q^{2k}} 
- \frac{q^2}{1-q^2}\frac{f(q^{2k+2})-f(0)}{q^{2k+2}} \big) =  f'(0)$ 
if the deri\-va\-tive of $f$ in $0$ exists. In that case, 
$\nabla_{\! q^2}\psi(y^2)$ defines a bounded operator on $\lN$. 
This together with \eqref{poly} motivates the following definition of an algebra of differentiable 
functions on the quantum disc: 
\[  \label{dFq}
\dFq := \left\{ \sum_{n=0}^N  z^n \hs f_n(y^2)  +  \sum_{n=1}^M f_{-n}(y^2) \hs z^{*n} \,:\, M,N\in\N,\ \, f_j\in C^{(1)}(\spec(y^2))\right\} . 
\]
Observe that differentiability of  $f\in C^{(1)}(\spec(y^2))$ amounts to the differentiability of $f$ in $0$ which is the only  accumulation point  
of $\spec(y^2)$. Using that $f(y) \in \dFq$ if and only if $f(q^ky) \in \dFq$  for any $k\in \Z$, 
it is easily seen that $\dFq $ is a *-algebra of bounded operators on $\lN$ containing $\Uq$. 
It follows from the (twisted) Leibniz rules 
\begin{align*}
&y^{-2} [\zeta, z^n\hs f(y) ] = y^{-2} [\zeta, z^n] f(y) + q^{-2n} z^n  y^{-2} [\zeta, f(y)], \\ 
&y^{-2} [\zeta,  f(y) z^{*n}  ] = y^{-2} [\zeta, f(y)]  z^{*n} +  f(y)   y^{-2} [\zeta, z^{*n} ], \ \ \, f\in C^{(1)}(\spec(y^2)), \ \,\zeta\in \{z,z^*\}, \ \,n\in\N, 
\end{align*} 
that the actions of $\dz$ and $\dbz$  given by the expressions in \eqref{dzdbz} are well defined on $\dFq$.

\section{Noncommutative integration} 

Let $s$ denote the unilateral shift operator acting on $\lN$ by $s\hs e_e = e_{n+1}$. Using the relations 
\[ \label{sz} 
z=s \hs \sqrt{1-q^2y^2},\quad z^*= s^*\hs \sqrt{1-y^2}  , \quad f(y) \hs s = s\hs f(qy) , \quad  s^*\hs f(y) = f(qy) \hs s^* , 
\]
where $f\in L_\infty(\spec(y)):= \{g: \spec(y) \ra \C: \text{bounded}\}$, one sees that each $\psi \in \dFq$ can be written 
$
\psi = \sum_{n=0}^N  s^n \hs g_n(y^2)  +  \sum_{n=1}^M g_{-n}(y^2) \hs s^{*n}
$
with continuous (but not necessarily differentiable) functions $g_j\in C(\spec(y^2))$.  
For a noncommutative integration theory on the quantum disc, it happens to be more convenient to work 
with the *-algebra of bounded operators 
\[ \label{Fq}
\Fq:= \left\{ \sum_{n=0}^N  s^n \hs g_n(y)  +  \sum_{n=1}^M g_{-n}(y) \hs s^{*n} \,:\, M, N\in\N,\ \, g_j\in  L_\infty(\spec(y))\right\}. 
\]
According to the  above remark, $\dFq\subset \Fq$.

We follow \cite{SSV, KueW, KW} and define for $\alpha>0$ 
\[ \label{Tr} 
\int_{\UD_q}^\alpha \ :\  \Fq \lra \C, \quad  \int_{\UD_q}^\alpha \psi := (1-q) \Tr_{\lN}(\psi\hs y^\alpha ) . 
\]
Since $y^\alpha$ is a positive compact operator for all $\alpha>0$, the weighted trace $\int_{\UD_q}^\alpha$ is a well-defined  positive functional 
on $\Fq$. 
As the trace over weighted shift operators vanishes, we have 
\[   \label{trace} 
\int_{\UD_q}^\alpha \psi = \int_{\UD_q}\left( \sum_{n=0}^N  s^n \hs g_n(y)  +  \sum_{n=1}^M g_{-n}(y) \hs s^{*n} \right) =  (1-q) \sum_{n\in\N} g_0(q^n) q^{\alpha n}. 
\]
In terms of the Jackson integral $\int_0^1 f(y)\dd_q y= (1-q) \sum_{n\in\N} f(q^n)q^n$,  the functional  $\int_{\UD_q}^\alpha$ may be written 
$$
 \int_{\UD_q}^\alpha \left( \sum_{n=0}^N s^n g_n(y) + \sum_{n=1}^M g_{-n} (y) s^{*n}\! \right) = 
 \int_0^1\! \int_{-\pi}^{\pi}  \left( \sum_{n=0}^N \e^{\im n\theta} g_n(y) + \sum_{n=1}^M g_{-n} (y) \e^{-\im n\theta}\right) 
 \dd \theta\, y^{\alpha -1} \dd_q y. 
$$

By \eqref{sz}, the commutation relation between $y^\alpha $ and functions from $\Fq$ can be expressed by 
the automorphism $\sigma^\alpha  : \Fq \lra \Fq$, 
\[ \label{s}
 \s^\alpha\Big(\sum_{n=0}^N  s^n \hs g_n(y)  +  \sum_{n=1}^M g_{-n}(y) \hs s^{*n}   \Big) 
=\sum_{n=0}^N  (q^{-\alpha}s)^n \hs g_n(y)  +  \sum_{n=1}^M g_{-n}(y) \hs (q^\alpha s)^{*n}. 
\]
Then $h\hs  y^\alpha= y^\alpha \s^\alpha(h)$ and hence, by the trace property, 
\[ \label{mod}
\int_{\UD_q}^\alpha gh = (1-q) \Tr_{\lN}(ghy^\alpha)  =  (1-q) \Tr_{\lN}(\s^\alpha (h)  g y^\alpha ) = \int_{\UD_q}^\alpha \s^\alpha(h)g 
\]
for all $f,g\in\Fq$. Note also that 
$$
  (\sa(h))^* = \s^{-\alpha}(h^*), \qquad h\in \Fq,
$$
where $\s^{-\alpha}$ denotes the inverse of $\sa$. 

Using $ \Tr_{\lN}(s^n s^{*k} f(y)y) =0$ if $k\neq n$, one easily verifies that $\int_{\UD_q}$ is faithful. 
Therefore 
\[     \label{ip} 
\ip{f}{g} := \int_{\UD_q}^\alpha f^* g, \qquad f,g\in\Fq, 
\]
defines  an inner product on $\Fq$. Its Hilbert space closure of will be denoted by $L_2(\UD_q)$. 

\section{Representations on the quantum disc} 

The left multiplication of  $\Uq$ on $\Fq$  defines a bounded *-representation of $\Uq$  on $L_2(\UD_q)$ 
since 
$$
\ip{xf}{g} =  \int_{\UD_q}^\alpha f^*\hs x^* g = \ip{f}{x^*g} , \qquad f,g\in\Fq,\quad x\in\Uq. 
$$

In \cite{KW}, representations of the opposite algebra $\Uq^\op$ play also  a crucial role in the definition 
of the Dirac operator. Recall that the opposite algebra $\A^\op$ of an algebra $\A$ is the vector space $\A$ with the opposite multiplication 
$ a^\op \, b^\op = (ba)^\op$. 
$\A^\op$ remains to be  a *-algebra if $\A$ is one. 
The right multiplication of $\Uq$ on $\Fq$ defines a representation of 
$\Uq^\op$ which is not a *-re\-pre\-sen\-ta\-tion because 
\[   \label{opstar}
\ip{x^\op\hs f}{g} = \ip{fx}{g} =   \int_{\UD_q}^\alpha x^* f^* g  =   \int_{\UD_q}^\alpha  f^* g\hs  \s^{-\alpha}( x^*) 
= \ip{f}{  \s^{-\alpha}( x^*)^\op g} 
\]
by \eqref{mod}. 

Our next aim is to construct an orthonormal basis for $L_2(\UD_q)$. To this end, 
we introduce the following notation:   As customary,  the symbol $\delta_{nm}$ denotes the Kronecker delta.
For  $k\in\Z$, let  $\delta_{q^{k}}  : \R \lra \{0,1\}$  be defined by 
 $$
\delta_{q^{k}}(t) :=    \left\{   \begin{array}{rl} 1, & t = q^{k} ,   \\ 0,  & t \neq q^{k} .      \end{array}         \right.   
$$
If $a$ is a densely defined operator, we set 
\[
   a^{\#k} := \left\{   \begin{array}{rl} a^k, &k \geq 0,   \\ a^{*k},  & k< 0 .    \end{array}         \right.   
\]  

Note that  $\delta_{q^{n}}$  is continuous on $\spec(y)$ and $\delta_{q^{n}}(y)$ is the orthogonal projection 
onto the one-dimen\-sional subspace $\gen\{e_n\} \subset \lN$. 
For any (measurable) function $f: \spec(y) \lra \C$, we have 
\[    \label{fd}
f(y)\hs  \delta_{q^{n}}(y)  =  \delta_{q^{n}}(y) \hs f(y)=  f(q^n)\hs \delta_{q^{n}}(y). 
\]
In particular, $\delta_{q^{m}}(y)  \delta_{q^{n}}(y) = \delta_{mn}\,\delta_{q^{n}}(y)$. 
Furthermore,  it follows from \eqref{sz}  that 
\[ \label{sds}
s \hs \delta_{q^{n}}(y)  =  \delta_{q^{n}}(q^{-1}y)  \hs s =  \delta_{q^{n+1}}(y)  \hs s, \quad 
s^* \hs \delta_{q^{n}}(y)  =  \delta_{q^{n}}(q\hs y)  \hs s^* =  \delta_{q^{n-1}}(y)  \hs s^*. 
\]
As a consequence, for all $k,n\in\N$ such that $n<k$, 
\[ \label{nk0}
s^{*k}  \hs \delta_{q^{n}}(y) =  \delta_{q^{n-k}}(y) \hs s^{*k}  =  0, \quad \delta_{q^{n}}(y)  \hs  s^{k}  =   s^{k}   \hs  \delta_{q^{n-k}}(y)= 0, 
\]
since $\delta_{q^{n-k}}(t) = 0$ on $\spec(y)$. 
From $ss^* = 1-  \delta_{q^{0}}(y)$, we get 
\[ \label{ssd} 
  s s^{*k}  \hs \delta_{q^{n}}(y) = (1-  \delta_{q^{0}}(y)) s^{*k-1}  \hs \delta_{q^{n}}(y) 
  =  s^{*k-1} \hs  (1-  \delta_{q^{k-1}}(y))  \hs \delta_{q^{n}}(y) =  s^{*k-1}  \hs \delta_{q^{n}}(y)
\]
as we may assume that $n> k-1$ by \eqref{nk0}.  Note that $s^{ \# k}   \hs  s^{\# l}  = s^{\# k+l}$ does not hold in general, 
for instance $s^{ \# 1}   \hs  s^{\# -1} = 1-  \delta_{q^{0}}(y) \neq s^{\# 0}$. However, from \eqref{ssd} and $s^* s=1$, we conclude that 
\[  \label{lk}
s^{\# l}  s^{\# k}  \hs \delta_{q^{n}}(y) =  s^{\# l+k}  \hs \delta_{q^{n}}(y)\quad \text{for all } \  k,l\in\Z, \ \ n\in\N. 
\]
Using \eqref{ip},  \eqref{lk}, \eqref{mod}, \eqref{fd} and \eqref{trace}, 
we compute  
\[   \label{onb}
\ip{  s^{\# k} \hs \delta_{q^{n}}(y)}{  s^{\# l} \hs \delta_{q^{m}}(y) } =  \int_{\UD_q}^\alpha\!  \delta_{q^{n}}(y) \hs  s^{ \# -k}   \hs  s^{\# l}  \delta_{q^{m}}(y) 
=  \int_{\UD_q}^\alpha \! s^{\# l-k}  \hs   \delta_{q^{m}}(y)  \hs    \delta_{q^{n}}(y) =(1\hsp - \hsp q) q^{\alpha n} \delta_{nm} \delta_{kl}. 
\]
The last equation is the key to constructing an orthonormal basis for $L_2(\UD_q)$. 
\begin{prop} \label{P}
For $n\in\N$ and $k\in\Z$ such that $ k \geq -n$, define 
\[ \label{deta}
  \eta_{nk} :=   \mbox{$\frac{q^{-\alpha n/2}}{\sqrt{1-q}}$} \hs s^{\# k}  \hs \delta_{q^{n}}(y). 
\]
Then $\{ \eta_{nk}  \,:  \,n\in\N, \ \, k\in\Z,\ \,  k \geq -n\} $ is an orthonormal basis for $L_2(\UD_q)$. 
\end{prop}
\begin{proof}
It follows from \eqref{onb} that the set defined in the proposition is orthonormal. Therefore it only remains 
to show that it is complete.  
As $L_2(\UD_q)$ is the closure of $\Fq$, it suffices to show that the elements $s^k\hs f(y)$ and $f(y) \hs s^{*k}$ 
can be expanded in this basis, where $f\in  L_\infty(\spec(y))$ and $k\in\N$. 
Since  $\delta_{q^{n}}(y)$ is the orthogonal projection 
onto  the eigenspace corresponding to the eigenvalue $q^n$ of the self-adjoint operator $y$, 
the spectral theorem gives $f(y) = \sum_{n\in\N} f(q^n) \hs \delta_{q^{n}}(y)$. 
Thus 
$s^k\hs f(y) = \sqrt{1-q}\hs  \sum_{n\in\N} q^{\alpha n/2}   \hs  f(q^n)  \hs  \eta_{nk}$  
and the series converges in $L_2(\UD_q)$ since $f$ is bounded and $\{q^{\alpha n/2}\}_{n\in\N} \in\lN$. 
Similarly, by \eqref{nk0}, 
$$
f(y) \hs s^{*k} =  \sum_{n\in\N} f(q^n) \hs \delta_{q^{n}}(y) \hs s^{*k}  =  \sum_{n\in\N} f(q^n) \hs s^{*k} \hs \delta_{q^{n+k}}(y) 
=  \sqrt{1-q}\hs q^{\alpha k/2}  \hs  \sum_{n\in\N} q^{\alpha n/2}   \hs  f(q^n)  \hs  \eta_{n+k,k}\, , 
$$ 
where the sequence of coefficients $\{q^{\alpha n/2} \hs  f(q^n)  \}_{n\in\N}$ belongs to  $\in\lN$.  
\end{proof} 

We will now compute the actions of our noncommutative multiplication and partial differential operators on this basis. 
To begin, 
\[ \label{yeta}
y \hs   \eta_{nk} =  \mbox{$\frac{q^{-\alpha n/2}}{\sqrt{1-q}}$} \hs y \hs  s^{\# k}  \hs \delta_{q^{n}}(y) 
= q^k \hs  \mbox{$\frac{q^{-\alpha n/2}}{\sqrt{1-q}}$} \hs  s^{\# k}\hs y   \hs \delta_{q^{n}}(y) 
=  q^{n+k} \hs  \mbox{$\frac{q^{-\alpha n/2}}{\sqrt{1-q}}$} \hs  s^{\# k}\hs \delta_{q^{n}}(y) 
= q^{n+k} \hs  \eta_{nk}\,,  
\]
where we used \eqref{sz} in the second equality and \eqref{fd} in the third. Similarly, 
\[
y^\op \hs   \eta_{nk} =   \eta_{nk} \hs y
= \mbox{$\frac{q^{-\alpha n/2}}{\sqrt{1-q}}$} \hs  s^{\# k}  \hs \delta_{q^{n}}(y) \hs y 
=  q^n\hs  \mbox{$\frac{q^{-\alpha n/2}}{\sqrt{1-q}}$} \hs  s^{\# k}  \hs \delta_{q^{n}}(y)
= q^{n} \hs  \eta_{nk}\,. 
\]
Writing $z= s\hs \sqrt{1-q^2 y^2}$, it follows from \eqref{yeta} and \eqref{lk} that 
\begin{align}\nonumber
z \hs   \eta_{nk}  &=   \mbox{$\frac{q^{-\alpha n/2}}{\sqrt{1-q}}$} \hs s\hs \sqrt{1-q^2 y^2} \hs  s^{\# k}  \hs \delta_{q^{n}}(y) 
=  \sqrt{1-q^{2(n+k+1)}} \hs \mbox{$\frac{q^{-\alpha n/2}}{\sqrt{1-q}}$} \hs s^{\# k+1}  \hs \delta_{q^{n}}(y) \\
& =  \sqrt{1-q^{2(n+k+1)}} \hs \eta_{n,k+1}\, ,   \label{zeta} 
\end{align}
and analogously, for $z^*= s^*\hs \sqrt{1-y^2}$, 
\begin{align}\nonumber
z^* \hs   \eta_{nk}  &=   \mbox{$\frac{q^{-\alpha n/2}}{\sqrt{1-q}}$} \hs s^*\hs \sqrt{1- y^2} \hs  s^{\# k}  \hs \delta_{q^{n}}(y) 
=  \sqrt{1-q^{2(n+k)}} \hs \mbox{$\frac{q^{-\alpha n/2}}{\sqrt{1-q}}$} \hs s^{\# k-1}  \hs \delta_{q^{n}}(y)\nonumber  \\
& =  \sqrt{1-q^{2(n+k)}} \hs \eta_{n,k-1}\, ,  \label{zetastar}
\end{align}
The actions of $z^\op$ and $z^{*\op}$ are calculated by applying \eqref{sds}, \eqref{fd} and \eqref{deta}:  
\begin{align}\nonumber
z^\op \hs   \eta_{nk}  =   \eta_{nk} \hs z 
 &= \mbox{$\frac{q^{-\alpha n/2}}{\sqrt{1-q}}$}  \hs  s^{\# k}  \hs \delta_{q^{n}}(y) \hs s\hs \sqrt{1-q^2 y^2} 
  = \mbox{$\frac{q^{-\alpha n/2}}{\sqrt{1-q}}$}  \hs  s^{\# k+1}  \hs \delta_{q^{n-1}}(y)\hs \sqrt{1-q^2 y^2} \nonumber\\ 
 &= q^{-\alpha/2}\hs  \sqrt{1-q^{2n}} \, \mbox{$\frac{q^{-\alpha (n-1)/2}}{\sqrt{1-q}}$} \hs s^{\# k+1}  \hs \delta_{q^{n-1}}(y)\nonumber \\ 
 &=  q^{-\alpha/2}\hs  \sqrt{1-q^{2n}}\, \eta_{n-1,k+1} \,,  \label{zetop} 
\end{align}
\begin{align}\nonumber
z^{*\op} \hs   \eta_{nk}  =   \eta_{nk} \hs z^* 
 &= \mbox{$\frac{q^{-\alpha n/2}}{\sqrt{1-q}}$}  \hs  s^{\# k}  \hs \delta_{q^{n}}(y) \hs s^* \hs \sqrt{1- y^2} 
  = \mbox{$\frac{q^{-\alpha n/2}}{\sqrt{1-q}}$}  \hs  s^{\# k-1}  \hs \delta_{q^{n+1}}(y)\hs \sqrt{1- y^2} \nonumber\\ 
 &= q^{\alpha/2}\hs  \sqrt{1-q^{2(n+1)}} \, \mbox{$\frac{q^{-\alpha (n+1)/2}}{\sqrt{1-q}}$} \hs s^{\# k+1}  \hs \delta_{q^{n+1}}(y)\nonumber  \\ 
 &=  q^{\alpha/2}\hs  \sqrt{1-q^{2(n+1)}}\, \eta_{n+1,k-1}\, .       \label{zetaopstar}
\end{align}
Note that $(z^\op)^* = q^{-\alpha} z^{*\op}  = (\s^{-\alpha} (z^*))^{\op}$ as observed in \eqref{opstar}. 

By Equation \eqref{dFq} and the remark following it, the basis vectors $ \eta_{nk} $ belong to $\dFq$. 
As eigenvectors of the self-adjoint operator $y$, they also belong to the domain of the unbounded operator $y^{-2}$. 
We will use the formulas in \eqref{dzdbz} to compute the actions of $\dz$ and $\dbz$ on $ \eta_{nk}$. 
From \eqref{yeta}, \eqref{zetastar} and \eqref{zetaopstar}, it follows that 
\begin{align} \nonumber
\dz\hs \eta_{nk} &=  \mbox{$\frac{1}{1-q^2}$}\hs  y^{-2}\hs(z^*\hs  \eta_{nk} -\eta_{nk}\hs  z^*) 
= \mbox{$\frac{1}{1-q^2}$}\hs  y^{-2}\hs(z^*\hs  \eta_{nk} -    z^{*\op}  \hs \eta_{nk}) \\
&= \mbox{$\frac{q^{-2(n+k)}}{1-q^2}$}\big(  
 q^{2}\hs \sqrt{1-q^{2(n+k)}} \hs \eta_{n,k-1} 
 -   q^{\alpha/2} \hs  \sqrt{1-q^{2(n+1)}}\, \eta_{n+1,k-1} \big) . 
\end{align} 
Further, by  \eqref{yeta}, \eqref{zeta} and \eqref{zetop},
\begin{align} \nonumber
\dbz\hs \eta_{nk} &=  \mbox{$\frac{-1}{1-q^2}$}\hs  y^{-2}\hs(z\hs  \eta_{nk} -\eta_{nk}\hs  z) 
= \mbox{$\frac{-1}{1-q^2}$}\hs  y^{-2}\hs(z \hs  \eta_{nk} -    z^{\op}  \hs \eta_{nk}) \\
&= -  \mbox{$\frac{q^{-2(n+k)}}{1-q^2}$}\big(  
 q^{-2}\hs \sqrt{1-q^{2(n+k+1)}} \hs \eta_{n,k+1} 
 -   q^{-\alpha/2} \hs  \sqrt{1-q^{2n}}\, \eta_{n-1,k+1} \big) .
\end{align} 

To compare the representations of $z$ and $z^*$ with the irreducible ones from \eqref{z},  
it is convenient to change the basis. For $n,k\in\N$,  set 
$e_{nk}:= \eta_{n,k-n}$. As an immediate consequence of Propo\-sition~\ref{P},  $\{e_{nk}: n,k\in\N\}$ is an 
orthonormal basis for $L_2(\UD_q)$ and the change of basis is given by the unitary operator 
$$
U:L_2(\UD_q)\lra L_2(\UD_q), \quad U\hs e_{nk} = \eta_{n,k-n}, \quad U^*\hs \eta_{nk} = e_{n,k+n}\,. 
$$
In the new basis, we have 
$$
y\hs e_{nk} = U^*\hs y \hs \eta_{n,k-n} = q^{k} \hs  U^*\hs \eta_{n,k-n} = q^{k}\hs e_{nk}\,, 
$$
\[       \label{ze}
z\hs e_{nk} = U^*\hs z \hs \eta_{n,k-n} 
=  \sqrt{1-q^{2(k+1)}} \hs U^*\hs \eta_{n,k-n+1} 
= \sqrt{1-q^{2(k+1)}} \hs e_{n,k+1} \,, 
\]
\[       \label{zestar}
z^*\hs e_{nk} = U^*\hs z^* \hs \eta_{n,k-n} 
=  \sqrt{1-q^{2k}} \hs U^*\hs \eta_{n,k-n-1} 
= \sqrt{1-q^{2k}} \hs e_{n,k-1} \,.
\]
In particular, on each Hilbert space 
\[ \label{Hn} 
\hH_n\,:=\, \gen\{ e_{nk} : k\in\N\}\,\cong \,\lN, 
\]
we recover the unique 
irreducible *-representations  \eqref{z} of the quantum disc $\Uq$. Further, for the opposite operators, one gets 
\[ \label{yop} 
y^\op \hs e_{nk} = U^*\hs y^\op \hs \eta_{n,k-n} = q^{n} \hs  U^*\hs \eta_{n,k-n} = q^{n}\hs e_{nk}\,, 
\] 
\[       \label{zeop}
z^\op  \hs e_{nk} = U^*\hs z^\op  \hs \eta_{n,k-n} 
= q^{-\alpha/2}\hs  \sqrt{1-q^{2n}}\, U^*\hs  \eta_{n-1,k-n+1}
= q^{-\alpha/2}\hs  \sqrt{1-q^{2n}}\,  e_{n-1,k}\,, 
\]
\[       \label{zestarop}  
z^{*\op}  \hs e_{nk} = U^*\hs z^{*\op}  \hs \eta_{n,k-n} 
= q^{\alpha/2}\hs  \sqrt{1-q^{2(n+1)}}\, U^*\hs  \eta_{n+1,k-n-1}
= q^{\alpha/2}\hs  \sqrt{1-q^{2(n+1)}}\, e_{n+1,k}\,. 
\] 
Now, setting $\zeta^\op := q^{\alpha/2} z^\op $, \,$\zeta^{*\op} := q^{-\alpha/2} z^{*\op}$ 
and 
\[ \label{Hop} 
\hH_k^\op\,:=\, \gen\{ e_{nk} : n\in\N\}\,\cong\, \lN, 
\] 
we obtain an irreducible 
*-representation of $\Uq^\op$ on $\hH_k^\op$, that is, $(\zeta^\op)^*= \zeta^{*\op}$ and 
\[  \label{zetaop} 
\zeta^\op\hs  \zeta^{*\op}  - q^2\hs  \zeta^{*\op}\hs  \zeta^{\op}  = 1- q^2. 
\] 
Note that we have of course 
\[ \label{fgop}
[f,g^\op] = f\hs g^\op - g^\op  \hs f =0 \quad \text{for all } \ f\in \Uq,  \ \ g^\op \in \Uq^\op, 
\] 
since the elements of $\Uq$ act only on the second index of $e_{nk}$ 
and  the elements of $\Uq^\op$ act only on the first. 

We compute the actions of $\dz$ and $\dbz$ on $e_{nk}$  by using the commutator representation 
\eqref{dzdbz} again. Thus 
$$
\dz\hs e_{nk} = \mbox{$\frac{1}{1-q^2}$}\hs  y^{-2}\hs(z^*\hs  e_{nk} -    z^{*\op}  \hs e_{nk}) 
=  \mbox{$\frac{q^{-2k}}{1-q^2}$}\hs  \big( 
q^2 \hs  \sqrt{1-q^{2k}} \hs e_{n,k-1}   - q^{\alpha/2}\hs  \sqrt{1-q^{2(n+1)}}\, e_{n+1,k} \big) ,
$$
$$
\dbz\hs e_{nk} = \mbox{$\frac{-1}{1-q^2}$}\hs  y^{-2}\hs(z\hs  e_{nk} -    z^{\op}  \hs e_{nk}) 
= -  \mbox{$\frac{q^{-2k}}{1-q^2}$}\hs  \big( 
q^{-2} \sqrt{1-q^{2(k+1)}} \hs e_{n,k+1}   - q^{-\alpha/2}  \sqrt{1-q^{2n}}\, e_{n-1,k} \big) . 
$$

Finally recall that \eqref{cd} defines a faithful *-representation of $\CSU$,  
and  consider the Hilbert space $L_2(\UD_q)\ob L_2(\S)$ with 
orthonormal basis $\{ e_{nkl}: n,k\in\N,\ l\in\Z\}$, where 
\[ \label{basis}
e_{nkl}:= e_{nk}\ot b_l =  
\mbox{$\frac{q^{-\alpha n/2}}{\sqrt{1-q}}$} \hs s^{\# k-n}  \hs \delta_{q^{n}}(y)  
\ot \mbox{$\frac{1}{\sqrt{2\pi}}$}  \E{l t} .
\] 
As in \eqref{cdot}, we define a *-representation of $\SUq$ on $L_2(\UD_q)\ob L_2(\S)$ 
by setting 
\[  \label{repcd}
 c:= y \ot u = \sqrt{1-zz^*} \ot u , \qquad d:= z \ot 1. 
\]
It follows from \eqref{ze}--\eqref{Hn} that $L_2(\UD_q)\,\bar\ot\, L_2(\S)$  decomposes into the direct sum 
$$
L_2(\UD_q)\,\bar\ot\,L_2(\S) = \bigoplus_{n\in\N} \hH_n\, \bar  \otimes \, L_2(\S) \cong  \bigoplus_{n\in\N} \lN\,\bar\ot\, \lZ, 
$$
and on each copy of $ \lN\,\bar\ot\, \lZ$, we recover the representation from \eqref{cd}. 
Hence the representation \eqref{repcd} decomposes into the infinite orthogonal sum of *-representations which are all 
unitarily equivalent to \eqref{cd}. In particular, \eqref{repcd} defines a faithful *-representation of $\SUq$ and the 
C*-closure of the *-algebra generated by $c$ and $d$ from \eqref{repcd} is isomorphic to $\CSU$. 

To obtain *-representations of the opposite algebras $\SUq^{\op}$ and $\CSU^{\op}$, the generators $c^\op$ and $d^\op$ 
have to satisfy the opposite relations of \eqref{rels1} and \eqref{rels2}. 
Analogous to $d= z \ot 1$ and $c = \sqrt{1- z z^* }  \ot  u$ with $z$ satisfying \eqref{qd} and $u$ being a unitary generator of $C(\S)$, 
a *-re\-pre\-sen\-ta\-tion of  $\SUq^{\op}$ is given by replacing $z$ by $\zeta^\op$ 
satisfying \eqref{zetaop} (the opposite relation of \eqref{qd}) and  $u$ by a unitary operator $u^\op$ with the same 
properties as $u$.  For $u^\op$ we may take $u^*$ which generates the same C*-algebra $C(\S)$ as $u$ does. 
The adjoint of $u$ is chosen for consistency because then $c$ and $d$ act as forward shifts,  and 
$c^\op$ and $d^\op$  act as backward shifts. For $\zeta^\op = q^{\alpha/2} z^\op$  and $\zeta^{*\op} = q^{-\alpha/2} z^{*\op}$,   
as defined in the paragraph after Equation \eqref{zestarop}, we get $\sqrt{1- \zeta^\op\hs  \zeta^{*\op}} =\sqrt{1- z^\op\hs  z^{*\op}} = y^\op$.  
Thus, setting 
\[ \label{repcdop} 
     c^\op := y^\op \ot u^* = \sqrt{1- \zeta^\op\hs  \zeta^{*\op}}\ot u^*,  \qquad   d^\op  :=  \zeta^\op \ot 1 , 
\]  
yields a *-representation of $\SUq^{\op}$ on $L_2(\UD_q)\,\bar\ot\,L_2(\S)$. 
Moreover, with $\hH_k^\op$ defined in \eqref{Hop}, we have the Hilbert space decomposition 
$$
L_2(\UD_q)\,\bar\ot\,L_2(\S) = \bigoplus_{k\in\N} \hH_k^\op \, \bar  \otimes \, L_2(\S) \cong  \bigoplus_{k\in\N} \lN\,\bar\ot\, \lZ, 
$$ 
and, by  \eqref{yop}--\eqref{zestarop}  and \eqref{u},  
$c^\op$ and $d^\op$  act  on each 
$\hH_{k_0}^\op \, \bar  \otimes \, L_2(\S) =\gen \{ e_{nk_0} \ot b_l : n\in\N, \ l\in\Z \}$ by 
\[  \label{cdop}
 c^\op\hs (e_{nk_0} \ot b_l) =  q^n \hs e_{nk_0} \ot b_{l-1}, \qquad
 d^\op \hs (e_{nk_0} \ot b_l) = \sqrt{1 - q^{2(n-1)}} \hs e_{n-1,k_0} \ot b_l\,. 
\]
Note that \eqref{cdop} are the adjoint relations of \eqref{cd}, and that 
taking the opposite relations of \eqref{rels1} and \eqref{rels2} amounts to interchanging $c$ and $d$ with 
their adjoints $c^*$ and $d^*$. Therefore, as much as \eqref{cd} defines a faithful *-representation 
of $\SUq$ and $\CSU$, \eqref{cdop} yields a faithful *-repre\-sen\-tation  of 
$\SUq^{\op}$ and $\CSU^{\op}$ and so does their direct sum representation on $L_2(\UD_q)\,\bar\ot\,L_2(\S)$. 
Finally we remark that the operators $c$ and $d$ from  \eqref{repcd} commute with the 
operators $c^\op$ and $d^\op$ from  \eqref{repcdop} since $u$ and $u^*$ belong to the commutative C*-algebra $C(\S)$,  
and the operators in the left factor of the tensor products commute by \eqref{fgop}.  
As a consequence, the representations of $\CSU$ and $\CSU^\op$ on $L_2(\UD_q)\,\bar\ot\,L_2(\S)$ 
commute. 

For the convenience of the reader, we finish the paper by 
collecting the most important formulas in a final theorem. 

\begin{thm}  Let $z,z^*\in B(\lN)$  denote the generators of the quantum disc $\Uq$ given in \eqref{z} 
and let $y = \sqrt{1-z\hs z^*}$. 
With $\Fq$ defined in \eqref{Fq}, the *-algebra  $\Fq \ot L_\infty(\S)$ 
is considered an algebra of bounded functions on 
a local  chart for quantum SU(2) in the following sense: 
Define an inner product on $\Fq \ot L_\infty(\S)$ by 
$$
\ip{f\ot \phi}{g\ot \psi} \, := \, \int_{\UD_q}^\alpha \! f^*g\, \int_{\S}\! \bar \phi \hs \psi \hs \dd \lambda 
\, = \, (1\hsp -\hsp q) \Tr_{\lN}(f^*g\hs y^\alpha )\int_{\S}\! \bar \phi \hs \psi \hs \dd \lambda , \quad \alpha >0, 
$$
where $\lambda$ stands for the Lebesgue measure on $\S$. 
The Hilbert space completion of $\Fq \ot L_\infty(\S)$  will be denoted by $L_2(\UD_q)\,\bar\ot\, L_2(\S)$. 
Furthermore, with  $\dFq$ described in \eqref{dFq}, 
the *-sub\-algebra $\dFq\ot C^{(1)}(\S)$ of  $\Fq \ot L_\infty(\S)$  is  viewed as an algebra of differentiable functions on the local chart. 
The Hilbert space $L_2(\UD_q)\,\bar\ot\, L_2(\S)$ has an orthonormal basis of 
differentiable functions $\{ e_{nkl} :  n,k\in\N,\ l\in\Z\} \subset \dFq\ot C^{(1)}(\S)$, where 
$$
e_{nkl} = 
\mbox{$\frac{q^{-\alpha n/2}}{\sqrt{1-q}}$} \hs s^{\# k-n}  \hs \delta_{q^{n}}(y)  
\ot \mbox{$\frac{1}{\sqrt{2\pi}}$}  \E{l t} .
$$
The left multiplication $x\,(f\ot \psi) := xf\otimes \psi$, \,$x\hsp\in\hsp\Uq$, \,$f\ot \psi\hsp\in\hsp \Fq \ot L_\infty(\S)$,  
defines a\linebreak[2]  bounded *-repre\-sen\-tation of $\Uq$ 
 on $L_2(\UD_q)\,\bar\ot\, L_2(\S)$.  On the above orthonormal basis, the opera\-tors $z$, $z^*$ and $y$ act by 
 $$
z \hs e_{nkl}  =   \sqrt{1-q^{2(k+1)}} \hs e_{n,k+1,l} \,, \quad z^* \hs e_{nkl}  =   \sqrt{1-q^{2k}} \hs e_{n,k-1,l}\,,\quad 
y\hs e_{nkl} = q^k \hs e_{nkl} \,. 
 $$
 The right multiplication $x^\op\,(f\ot \psi) := fx\otimes \psi$, \,$x \in \Uq^\op$, \,$f\ot \psi \in \Fq \ot L_\infty(\S)$, 
defines a bounded repre\-sen\-tation of $\Uq^\op$ 
 on $L_2(\UD_q)\,\bar\ot\, L_2(\S)$ which is not a *-representation. 
 The actions of  $z^\op$, $z^{*\op} = q^{\alpha }(z^\op)^* $ and $y^\op = (y^\op)^*$ are given by 
 $$
z^\op \hs e_{nkl}   = q^{-\alpha/2}\hs  \sqrt{1-q^{2n}}\,  e_{n-1,kl}\,, \quad 
z^{*\op} \hs e_{nkl}   = q^{\alpha/2}\hs  \sqrt{1-q^{2(n+1)}}\,  e_{n+1,kl}\,, \quad 
y^\op \hs e_{nkl} = q^n \hs e_{nkl}\, . 
 $$
 Let $u\in C(\S)$, $u(\E{t})=\E{t}$, denote the unitary generator of $C(\S)$. Setting 
 \begin{align*}
  c&:= y \ot u = \sqrt{1-zz^*} \ot u , & d&:= z \ot 1,  \\
  c^\op &:= y^\op \ot u^* = \sqrt{1- z^\op\hs  z^{*\op}}\ot u^*, &  d^\op  &:= q^{\alpha/2} z^\op \ot 1 , 
 \end{align*}
yields  commuting *-representations of $\CSU$ and $\CSU^\op$ on $L_2(\UD_q)\,\bar\ot\, L_2(\S)$ 
 given by left and right multiplication, respectively.  
 On basis vectors, these representations read 
 \begin{align*}
 c \hs e_{nkl} &= q^k\hs e_{nk,l+1}\, ,    & d \hs e_{nkl}  &=   \sqrt{1-q^{2(k+1)}} \, e_{n,k+1,l}\, , \\
c^\op \hs e_{nkl} &= q^n \hs e_{nk,l-1}\, ,  &  d^\op \hs e_{nkl}   &=   \sqrt{1-q^{2n}}\,  e_{n-1,kl}\, . 
 \end{align*}
 The partial derivatives $\dz$, $\dbz$ and $\dt$ act on differentiable functions $g\ot \phi\in \dFq\ot C^{(1)}(\S)$ by 
 $\dz (g\ot \phi) = \dz g \ot \phi$, \,$\dbz (g\ot \phi) = \dbz g \ot \phi$ and $\dt (g\ot \phi) = g \ot \dt \phi$. 
 On basis elements, one obtains 
 \begin{align*}
\dz\hs e_{nkl} &=  \mbox{$\frac{q^{-2k}}{1-q^2}$}\hs  \big( 
q^2 \hs  \sqrt{1-q^{2k}} \hs \eta_{n,k-1,l}   - q^{\alpha/2}\hs  \sqrt{1-q^{2(n+1)}}\, e_{n+1,kl} \big) ,\\
\dbz\hs e_{nkl} &= -  \mbox{$\frac{q^{-2k}}{1-q^2}$}\hs  \big( 
q^{-2} \sqrt{1-q^{2(k+1)}} \hs \eta_{n,k+1,l}   - q^{-\alpha/2}  \sqrt{1-q^{2n}}\, e_{n-1,kl} \big) , \\
\dt \hs e_{nkl} &= \im \hs l\hs e_{nkl} \, . 
 \end{align*}
\end{thm} 

Note that the representations of $\CSU$ and $\CSU^\op$ do not depend on $\alpha$ 
but the actions of $\dz$ and $\dbz$ do so. 

\section*{Acknowledgements}  
The author thanks Ulrich Kr\"ahmer and Andrzej Sitarz for interesting discussions on the subject, and 
acknowledges financial support from the Polish Government grant 3542/H2020/2016/2, 
the EU funded grant H2020-MSCA-RISE-2015-691246-QUANTUM DYNAMICS and 
CIC-UMSNH. 


\end{document}